\documentclass[12pt]{amsart}

\usepackage[colorlinks=true, pdfstartview=FitV, linkcolor=blue, citecolor=blue, urlcolor=blue, breaklinks=true]{hyperref}
\usepackage{amsmath,amsfonts,amssymb,amsthm,amscd,comment,paralist,stmaryrd,etoolbox,mathtools}
\usepackage[usenames]{color}
\usepackage{mdwlist}

\usepackage[all]{xy}
\usepackage{tikz}
\newtheorem{theo}{Theorem}[section]
\newtheorem*{theo*}{Theorem}
\newtheorem{prop}[theo]{Proposition}
\newtheorem{lem}[theo]{Lemma}
\newtheorem{cor}[theo]{Corollary}

\theoremstyle{definition}
\newtheorem{defn}[theo]{Definition}
\newtheorem{rem}[theo]{Remark}
\newtheorem{remark}[theo]{Remark}
\newtheorem{example}[theo]{Example}

\DeclareMathOperator{\Hom}{Hom}
\DeclareMathOperator{\coker}{coker}
\DeclareMathOperator{\Rep}{rep}
\DeclareMathOperator{\im}{im}
\DeclareMathOperator{\Mat}{Mat}

\DeclareMathOperator{\rep}{rep}
\DeclareMathOperator{\wtrep}{wt-rep}
\DeclareMathOperator{\id}{id}
\DeclareMathOperator{\Md}{-Mod}

\newcommand{\R}{\mathbb{R}}
\newcommand{\Q}{\mathbb{Q}}
\newcommand{\N}{\mathbb{N}}
\newcommand{\C}{\mathbb{C}}
\newcommand{\Z}{\mathbb{Z}}

\newcommand{\ml}{\mathcal}
\newcommand{\mk}{\mathfrak}

\begin{document}

\title{Quivers and Three Dimensional Lie Algebras}
\author{Jeffrey Pike}
\address{Department of Mathematics and Statistics, University of Ottawa, Ottawa, Ontario K1N 6N5, Canada}
\email{jpike061@uottawa.ca}

\thanks{The author was supported by the NSERC Discovery Grant of Alistair Savage}

\begin{abstract}
We study a family of three-dimensional Lie algebras $L_\mu$ that depend on a continuous parameter $\mu$.  We introduce certain quivers, which we denote by $Q_{m,n}$ $(m,n \in \Z)$ and $Q_{\infty \times \infty}$, and prove that idempotented versions of the enveloping algebras of the Lie algebras $L_{\mu}$ are isomorphic to the path algebras of these quivers modulo certain ideals in the case that $\mu$ is rational and non-rational, respectively. We then show how the representation theory of the quivers $Q_{m,n}$ and $Q_{\infty\times\infty}$ can be related to the representation theory of quivers of affine type $A$, and use this relationship to study representations of the Lie algebras $L_\mu$.  In particular, though it is known that the Lie algebras $L_\mu$ are of wild representation type, we show that if we impose certain restrictions on weight decompositions, we obtain full subcategories of the category of representations of $L_\mu$ that are of finite or tame representation type.
\end{abstract}

\subjclass[2010]{Primary: 17B10, 16G20, Secondary: 22E47}

\keywords{Lie algebra, quiver, path algebra, preprojective algebra, representation}

\maketitle
\thispagestyle{empty}
\tableofcontents
\numberwithin{equation}{section}

%%%%%%%%%%%%%%%%%%%%%%%%%%%%%%%%%%%

\section*{Introduction}
In the late 20$^{\text{th}}$ century, the mathematician Pierre Gabriel discovered a beautiful relationship between the root systems of Lie algebras and the representations of \emph{quivers} \cite{Gabriel}, which are directed graphs.  This result fuelled further research of the connection between Lie theory and quivers, as the correspondence between the two allows for a more intuitive and geometric means of studying Lie algebras.  In the present work, we study the representation theory of a particular collection of Lie algebras by first relating them to certain quivers.

When working over the complex numbers, one can completely classify all Lie algebras of dimension three, up to isomorphism.  The possibilities are (see for example \cite[Chapter 3]{ErdLie}):
\begin{enumerate}
\item the three-dimensional abelian Lie algebra,
  \item the direct sum of the unique nonabelian two-dimensional Lie algebra with the one-dimensional Lie algebra,
  \item the Heisenberg algebra,
  \item the Lie algebra $\mk{sl}_2(\C)$,
    \item the Lie algebra with basis $\{x,y,z\}$ and commutation relations $[x,y]=y$, $[x,z]=y+z$, $[y,z]=0$.
  \item the Lie algebras $L_{\mu}$, $\mu \in \C^*=\C\setminus\{0\}$ (see Section \ref{liealg}).
\end{enumerate}
The representation theory of the first four Lie algebras in the above list is well understood, while the representation theory of the last two is less so.  In particular, for generic $\mu$, very little is known about the representation theory of the Lie algebras $L_{\mu}$, and so in order to obtain a better understanding of the representation theory of three-dimensional Lie algebras over the complex numbers, we first require a better understanding of this particular family.  Due to a result of Makedonskyi \cite[Theorem 3]{Mak}, it is known that the Lie algebras $L_{\mu}$ are of wild representation type.  However, in the current paper, we will exploit the relationship between quivers and Lie algebras, to draw several conclusions about the representations of these Lie algebras.

The \emph{Euclidean group} is the group of isometries of $\R^2$ having determinant 1, and the \emph{Euclidean algebra} is the complexification of its Lie algebra.  The Euclidean group is one of the oldest and most studied examples of a group: it was studied implicitly even before the notion of a group was formalized, and it has applications not only throughout mathematics but in quantum mechanics, relativity, and other areas of physics as well.  In \cite[Theorem 4.1]{SavEuc}, it is shown that the category of representations of the Euclidean algebra admitting weight space decompositions is equivalent to the category of representations of the preprojective algebras of quivers of type $A_{\infty}$.   In the current paper, we show that the category of representations of $L_{\frac{m}{n}}$ ($m,n \in \Z$, $n\neq 0$) admitting weight space decompositions can be embedded inside the category of representations of the preprojective algebra of the affine quiver of type $A^{(1)}_{m+n}$, where by convention $A^{(1)}_{0}$ denotes the quiver of type $A_{\infty}$.  It can be shown that the Euclidean algebra is isomorphic to the Lie algebra $L_{-1}$, and so if $\mu =-1$ then $m=-n$ so that this agrees with what is presented in \cite{SavEuc}.  Thus the current work can be thought of as a generalization of some of the results of that paper.  Analogous to the rational case, we also show how the representations with weight space decompositions of $L_\mu$, $\mu \in \C\setminus \Q$, form a subcategory of the category of representations of the preprojective algebra of the quiver of type $A_{\infty}$.

We begin by defining the \emph{modified enveloping algebras} of the Lie algebras $L_\mu$, denoted $\widetilde U_\mu$, and we note that category of representations of $\widetilde U_\mu$ is equivalent to the category of representations of $L_\mu$ admitting weight space decompositions.  We then introduce certain quivers, denoted $Q_{m,n}$ $(m,n\in\Z)$ and $Q_{\infty\times\infty}$, and show that the algebras $\widetilde U_\mu$ are isomorphic to the path algebras of $Q_{m,n}$ and $Q_{\infty\times\infty}$ modulo certain ideals in the case that $\mu$ is rational and non-rational, respectively (see Proposition \ref{ratiso}).  We use the theory of quiver morphisms to relate the representation theory of the quivers $Q_{m,n}$ and $Q_{\infty\times\infty}$ to the representation theory of affine quivers of type $A$, which is well understood.  In the case that $\mu \in \Q$, the main result is the following (Theorems \ref{m=1} and \ref{mneq1}):
\begin{theo*}
Let $m,n \in \Z$, $\gcd(m,n)=1$, and let $a,b \in \Z$ be such that $0\leq b-a<m$.  Let $\ml C^{m,n}_{a,b}$ denote the full subcategory of $\widetilde{U}_{\mu}\Md$ consisting of modules $V$ such that $V_k=0$ whenever $k<a$ or $k>b$, where $V_k$ denotes the $k^{\text{th}}$ weight space of $V$.  Then $\ml C^{m,n}_{a,b}$ is of finite representation type when $n\neq 1$, and $\ml C^{m,1}_{a,b}$ is of tame representation type.  
\end{theo*}
When $\mu \in \C\setminus \Q$, we introduce a $\Z$ action on $\widetilde U_\mu \Md$ and our main result is the following (Corollary \ref{irratreps}):
\begin{theo*}
Let $A$ be a finite subset of $\Z$ with the property that $A$ does not contain any five consecutive integers.  Then there are a finite number of $\Z$-orbits of isomorphism classes of indecomposable $\widetilde U_\mu$-modules $V$ such that $V_{ij}=0$ whenever $i-j\notin A$, where $V_{ij}$ denotes the $(i,j)$ weight space of $V$.  
\end{theo*}

The organization of the paper is as follows.  In Section \ref{liealg} we introduce the family $L_{\mu}$ of 3-dimensional Lie algebras that will be studied in the rest of the paper.  We then describe the (modified) universal enveloping algebras $\widetilde U_{m,n}$ and $\widetilde U_{\mu}$ associated with the Lie algebras $L_{\mu}$.  In Section \ref{quiv} we recall some basic notions from the theory of quivers and their representations, and in Section \ref{quimorph} we develop the theory of morphisms between quivers.  Finally, in Section \ref{application}, we establish a relationship between the representation theory of the Lie algebras $L_{\mu}$ and the representation theory of the quivers $Q_{m,n}$ and $Q_{\infty\times\infty}$, and use some of the results of Section \ref{quimorph} to study these quivers.

\bigskip

\textbf{Acknowledgements.} The research for this paper began in the Summer of 2012 at the University of Ottawa under the supervision of Alistair Savage.  We would like to thank him not only for introducing us to the topics of this paper, but also for all the interesting discussions and for his insightful guidance over the course of this project.

\bigskip

\textbf{Notation.}  Throughout this paper, all vector spaces and linear maps will be over $\C$.  Given a $\C$-algebra $A$, we will take the term \emph{module over $A$} to mean \emph{left} module over $A$.  The category of (left) modules over $A$ is denoted $A\Md$.  By the usual abuse of language, we will use the terms module and representation interchangeably.

%%%%%%%%%%%%%%%%%%%%%%%%%%%%%%%%%%%

\section{The Lie Algebras $L_{\mu}$}\label{liealg}
In this section we introduce the reader to the family of three-dimensional Lie algebras that will be the main focus of this paper.  Given a complex number $z \in \C$, we will denote by $z_R$ the real part of $z$ and by $z_I$ the imaginary part.

Let $\mu \in \C^*=\C\setminus\{0\}$ and let $L_\mu$ denote the three-dimensional Lie algebra with basis $\{\alpha_1,\alpha_2,\beta\}$ and commutation relations
\begin{equation}\label{relations}
  [\beta,\alpha_1]=\alpha_1,\quad \quad [\beta,\alpha_2]=\mu \alpha_2, \quad \quad [\alpha_1,\alpha_2]=0.
\end{equation}
It is straightforward to check that $L_\mu \cong L_\nu$ if and only if $\mu=\nu$ or $\mu=\nu^{-1}$.  We will separate our discussion of this family of Lie algebras into two cases: $\mu \in \Q^*$ and $\mu \in \C\backslash\Q$. 

Let $\mu \in \Q^*$, say $\mu = \frac{m}{n}$ for $m,n\in \Z$ with $n\neq 0$ and $\gcd(m,n)=1$.  We will change bases so that the commutation relations become
\begin{equation}\label{relationsq}
  [\beta,\alpha_1]=m\alpha_1,\quad \quad [\beta,\alpha_2]=n\alpha_2, \quad \quad [\alpha_1,\alpha_2]=0.
\end{equation}
Then for any indecomposable $L_{\mu}$-module $V$ in which $\beta$ acts semisimply, $\beta$ will act with eigenvalues from a set of the form $\gamma + \Z$ for some $\gamma \in \C$ with $0\leq \gamma_R <1$.  We will write $V_{\lambda}$ to represent the eigenspace of $\beta$ with eigenvalue $\lambda$. This gives the following eigenspace decomposition:
\begin{equation}\label{weightspace1} \textstyle
V=\bigoplus_{k\in\Z}V_{ k}, \qquad V_{k}=\{v \in V \ |\ \beta \cdot v=(\gamma+k) v\}.
\end{equation}

Let $U_{m,n}$\label{umn} be the universal enveloping algebra of $L_\mu$ and let $U^0, U^1, U^2$ be the subalgebras generated by $\beta, \alpha_1, \alpha_2$ respectively. Then, by the PBW Theorem, we have
\begin{equation} \label{univalg}
U_{m,n} \cong U^1 \otimes U^0 \otimes U^2 \quad (\text{as vector spaces}).
\end{equation}

From \eqref{relationsq} we obtain the following relations in the universal enveloping algebra:
\begin{equation}
\beta\alpha_1-\alpha_1\beta=m\alpha_1, \qquad \beta\alpha_2-\alpha_2\beta=n\alpha_2, \qquad \alpha_1\alpha_2=\alpha_2\alpha_1 \label{Urel}.
\end{equation}

Following \cite[Section 2]{SavEuc} we will consider the \index{modified enveloping algebra}modified enveloping algebra $\widetilde{U}_{m,n}$ of $L_\mu$ by replacing $U^0$ with a sum of 1-dimensional algebras:
\begin{equation} \textstyle \label{modenv}
\widetilde{U}_{m,n}=U^1 \otimes \left( \bigoplus_{k \in \Z} \C a_{k} \right) \otimes U^2.
\end{equation}

Multiplication in the modified enveloping algebra is given by
\begin{align}
a_ka_{\ell}&=\delta_{k\ell}a_k, \nonumber \\
\alpha_1a_k=a_{k+m}\alpha_1 &, \quad \alpha_2a_k=a_{k+n}\alpha_2,\label{modrel} \\
\alpha_1\alpha_2a_k&=\alpha_2\alpha_1a_k , \nonumber
\end{align}
where $k, \ell \in \Z$.  The element $a_k$ is an idempotent that will act on any module $V$ as projection onto the $k$-th weight space $V_k$.  For $\gamma \in \C$ with $0\leq \gamma_R <1$, denote by $\wtrep_\gamma(U_{m,n})$ the category of representations of $U_{m,n}$ on which $\beta$ acts semisimply with eigenvalues from the set $\gamma +\Z$.  Then we have an equivalence of categories $\wtrep_\gamma(U_{m,n})\cong \rep(\widetilde U_{m,n})$.  

When $\mu \in \C \setminus \Q$, $\beta$ will act on indecomposable $L_{\mu}$-modules $V$ admitting weight space decompositions with eigenvalues of the form $\gamma +k$ for some $\gamma \in \C$ with $0\leq \gamma_R < \min(1,\mu_R)$ and $0\leq \gamma_I <\mu_I$, where $k \in \Z+\Z\mu$. Therefore, we have the weight space decomposition
\begin{equation}\label{weightspace2}
V=\bigoplus_{i+j\mu \in\Z+\Z\mu}V_{ij}, \quad V_{ij}=\{v \in V \ |\ \beta \cdot v=(\gamma+i+j\mu) v\}.
\end{equation}

When $\mu \in \C \setminus \Q$, we get the same decomposition of the universal enveloping algebra as found in \eqref{univalg}. However, the relations found in \eqref{Urel} become
\begin{equation}
\beta\alpha_1-\alpha_1\beta=\alpha_1, \qquad \beta\alpha_2-\alpha_2\beta=\mu\alpha_2, \qquad \alpha_1\alpha_2=\alpha_2\alpha_1 \label{Urelira}.
\end{equation}

Again we consider the \index{modified enveloping algebra}modified enveloping algebra, $\widetilde U_{\mu}$, in this case given by
\begin{equation} \textstyle \label{modenv2}
\widetilde{U}_{\mu}=U^1 \otimes \left( \bigoplus_{k \in \Z + \mu \Z} \C a_{k} \right) \otimes U^2.
\end{equation}

Multiplication in this modified enveloping algebra is given by
\begin{align}
a_ka_{\ell}&=\delta_{k\ell}a_k, \nonumber \\
\alpha_1a_k=a_{k+1}\alpha_1 &, \quad \alpha_2a_k=a_{k+\mu}\alpha_2,\label{modrelirra} \\
\alpha_1\alpha_2a_k&=\alpha_2\alpha_1a_k , \nonumber
\end{align}
where $k, \ell \in \Z + \mu \Z$.  Since $\mu \in \C \setminus \Q$, we have $\Z + \mu \Z \cong \Z \times \Z$, so we can reindex the projections $a_k$ by defining $a_{ij}=a_{i+j\mu}$.  In this notation the modified enveloping algebra has the form
\begin{equation}\textstyle \label{modenv3}
\widetilde{U}_{\mu}=U^1 \otimes \left( \bigoplus_{i,j \in \Z}  \C a_{ij}  \right) \otimes U^2.
\end{equation}
The multiplication is given by:
\begin{align}
a_{ij}a_{st}&=
\begin{cases}
 a_{ij}, & \text{if }i=s, j=t, \\
0, & \text{otherwise,}
\end{cases}
\nonumber \\
\alpha_1a_{ij}=a_{(i+1)j}\alpha_1 &, \quad \alpha_2a_{ij}=a_{i(j+1)}\alpha_2,\label{modrelirra2} \\
\alpha_1\alpha_2a_{ij}&=\alpha_2\alpha_1a_{ij} \nonumber.
\end{align}

\begin{rem}\label{remark1}
The importance of this reindexing is that we have eliminated any dependence on $\mu$. This shows that when $\mu \in \C \setminus \Q$ the modified enveloping algebras of all the Lie algebras $L_{\mu}$ are isomorphic.
\end{rem}

Let $U_\mu$ denote the universal enveloping algebra of $L_\mu$, and for any $\gamma \in \C$ with $0\leq \gamma_R < \min(1,\mu_R)$ and $0\leq \gamma_I <\mu_I$, denote by $\wtrep_\gamma(U_\mu)$ the category of representations of $U_{\mu}$ on which $\beta$ acts semisimply with eigenvalues from the set $\gamma +\Z+\Z\mu$.  Then we have an equivalence of categories $\wtrep(U_{\mu}) \cong \rep(\widetilde U_\mu)$.  Note that the categories of weight representations of $L_{\mu}$ are all equivalent whenever $\mu \in \C \setminus \Q$ (see Remark \ref{remark1}).

%%%%%%%%%%%%%%%%%%%%%%%%%%%%%%%%%%%

\section{Quivers and the Path Algebra}\label{quiv}

In this section we recall some basic notions related to quivers.  We introduce quivers, the path algebra of a quiver, and representations of quiver, and mention the equivalence of categories between modules over the path algebra and representations of a quiver.

A \emph{quiver} $Q$ is a 4-tuple $(X,A,t,h)$, where $X$ and $A$ are sets, and $t$ and $h$ are functions from $A$ to $X$.  The sets $X$ and $A$ are called the \emph{vertex} and \emph{arrow} sets respectively.  If $\rho \in A$, we call $t(\rho)$ the \emph{tail} of $\rho$, and $h(\rho)$ the \emph{head}.  We can think of an element $\rho \in A$ as an arrow from the vertex $t(\rho)$ to the vertex $h(\rho)$.  We will often denote a quiver simply by $Q=(X,A)$, leaving the maps $t$ and $h$ implied.
\begin{example}[The quiver $Q_{m,n}$]\label{qmn}
Let $m,n \in \Z$ be nonzero integers and consider the quiver $Q_{m,n}=(\Z,A^{m,n})$ where $A^{m,n}=\{\rho_{j}^{k}\ |\ k \in \Z, \, j=1,2 \}$ is the arrow set. We define a map $\sigma:\{1,2\} \to \{m,n\}$ such that $\sigma(1)=m$ and $\sigma(2)=n$. Then $\rho_{j}^{k}$ is an arrow whose tail, $t(\rho_{j}^{k})$, is the vertex $k$ and whose head, $h(\rho_{j}^{k})$, is the vertex $k+\sigma(j)$.  If $\gcd(m,n) \neq 1$, then the quiver $Q_{m,n}$ decomposes into a disjoint union of quivers of the form $Q_{m',n'}$, where $\gcd(m',n')=1$.  Thus we may assume when dealing with the quivers $Q_{m,n}$ that $\gcd(m,n)=1$.
\begin{figure}[h]
\[
\begin{xy}
\POS(0,0) *\cir<3pt>{}="0"
\POS(10,0) *\cir<3pt>{}="1"
\POS(25,0) *\cir<3pt>{}="m"
\POS(35,0) *\cir<3pt>{}="m+1"
\POS(50,0) *\cir<3pt>{}="n"
\POS(60,0) *\cir<3pt>{}="n+1"
\POS(-25,0) *\cir<3pt>{}="-m"
\POS(-15,0) *\cir<3pt>{}="1-m"
\POS(-50,0) *\cir<3pt>{}="-n"
\POS(-40,0) *\cir<3pt>{}="1-n"
\ar @/^2pc/ "0";"m" <0pt>
\ar @/^2pc/ "1";"m+1" <0pt>
\ar @/_4pc/ "0";"n" <0pt>
\ar @/_4pc/ "1";"n+1" <0pt>
\ar @/^2pc/ "-m";"0" <0pt>
\ar @/^2pc/ "1-m";"1" <0pt>
\ar @/_4pc/ "-n";"0" <0pt>
\ar @/_4pc/ "1-n";"1" <0pt>
\POS(0,-5) *+{0}
\POS(10,-5) *+{1}
\POS(25,-5) *+{m}
\POS(35,-5) *+{m+1}
\POS(50,-5) *+{n}
\POS(60,-5) *+{n+1}
\POS(-25,-5) *+{-m}
\POS(-15,-5) *+{1-m}
\POS(-40,-5) *+{1-n}
\POS(-50,-5) *+{-n}
\POS(17.5,0) *+{\cdots}
\POS(42.5,0) *+{\cdots}
\POS(-7.5,0) *+{\cdots}
\POS(-32.5,0) *+{\cdots}
\POS(65,0) *+{\cdots}
\POS(-55,0) *+{\cdots}
\end{xy}
\]
\caption{The Quiver $Q_{m,n}$}
\end{figure}
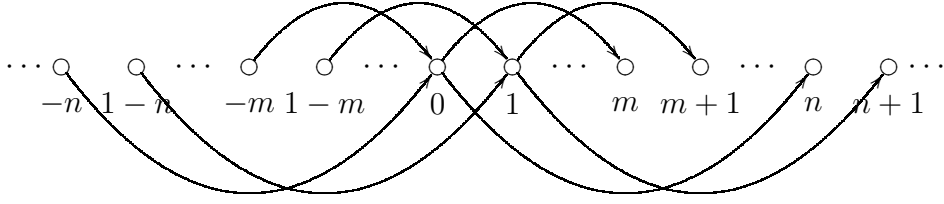
\end{example}
\begin{example}[The quiver $Q_{\infty \times \infty}$]\label{qinf}
We now consider the quiver $Q_{\infty \times \infty}=(\Z \times \Z,A_{\infty \times \infty})$ where the set of arrows is $A_{\infty \times \infty}=\{ \rho^{ij}_d \ |\  d \in \{1,2\}, (i,j) \in \Z \times \Z \}$.  We define the map $\theta : \{1,2\} \rightarrow \{(1,0) , (0,1)\}$ by $\theta(1)=(1,0)$ and $\theta(2)=(0,1)$.  Then $\rho^k_d$ is the arrow whose tail, $t(\rho^{ij}_d)$, is the vertex $(i,j)$ and whose head, $h( \rho^{ij}_d)$, is the vertex $(i,j)+\theta(d)$.
\begin{figure}[h2]
\begin{center}
\begin{tikzpicture}
\draw (0.1,0.1) circle [radius=0.1];
\draw (2.1,0.1) circle [radius=0.1];
\draw (4.1,0.1) circle [radius=0.1];
\draw (0.1,1.1) circle [radius=0.1];
\draw (2.1,1.1) circle [radius=0.1];
\draw (4.1,1.1) circle [radius=0.1];
\draw (0.1,2.1) circle [radius=0.1];
\draw (2.1,2.1) circle [radius=0.1];
\draw (4.1,2.1) circle [radius=0.1];
\draw [thick, ->] (0.2,0.1) -- (2.0,0.1);
\draw [thick, ->] (0.1,0.2) -- (0.1,1.0);
\draw [thick, ->] (2.2,0.1) -- (4.0,0.1);
\draw [thick, ->] (2.1,0.2) -- (2.1,1.0);
\draw [thick, ->] (4.1,0.2) -- (4.1,1.0);
\draw [thick, ->] (0.2,1.1) -- (2.0,1.1);
\draw [thick, ->] (2.2,1.1) -- (4.0,1.1);
\draw [thick, ->] (0.1,1.2) -- (0.1,2.0);
\draw [thick, ->] (2.1,1.2) -- (2.1,2.0);
\draw [thick, ->] (4.1,1.2) -- (4.1,2.0);
\draw [thick, ->] (0.2,2.1) -- (2.0,2.1);
\draw [thick, ->] (2.2,2.1) -- (4.0,2.1);
\draw [thick, ->] (0.1,-0.4) -- (0.1,0);
\draw [thick, ->] (2.1,-0.4) -- (2.1,0);
\draw [thick, ->] (4.1,-0.4) -- (4.1,0);
\draw [thick, ->] (0.1,2.2) -- (0.1,2.6);
\draw [thick, ->] (2.1,2.2) -- (2.1,2.6);
\draw [thick, ->] (4.1,2.2) -- (4.1,2.6);
\draw [thick, ->] (-0.6,0.1) -- (0,0.1);
\draw [thick, ->] (-0.6,1.1) -- (0,1.1);
\draw [thick, ->] (-0.6,2.1) -- (0,2.1);
\draw [thick, ->] (4.2,0.1) -- (4.8,0.1);
\draw [thick, ->] (4.2,1.1) -- (4.8,1.1);
\draw [thick, ->] (4.2,2.1) -- (4.8,2.1);
\node at (0.85,-0.2) {($-1$,$-1$)};
\node at (2.77, -0.2) {(0,$-1$)};
\node at (4.77, -0.2) {(1,$-1$)};
\node at (0.7,0.8) {($-1$,0)};
\node at (2.7, 0.8) {(0, 0)};
\node at (4.7, 0.8) {(1, 0)};
\node at (0.7,1.8) {($-1$,1)};
\node at (2.7, 1.8) {(0, 1)};
\node at (4.7, 1.8) {(1, 1)};
\node at (5.3,2.1) {. . .};
\node at (5.3,1.1) {. . .};
\node at (5.3,0.1) {. . .};
\node at (-1.1,2.1) {. . .};
\node at (-1.1,1.1) {. . .};
\node at (-1.1,0.1) {. . .};
\node at (0.1,-1) {.};
\node at (0.1,-0.6) {.};
\node at (0.1,-0.8) {.};
\node at (2.1,-0.6) {.};
\node at (2.1,-1) {.};
\node at (2.1,-0.8) {.};
\node at (4.1,-0.6) {.};
\node at (4.1,-1) {.};
\node at (4.1,-0.8) {.};
\node at (4.1,2.8) {.};
\node at (4.1,3) {.};
\node at (4.1,3.2) {.};
\node at (2.1,2.8) {.};
\node at (2.1,3) {.};
\node at (2.1,3.2) {.};
\node at (0.1,2.8) {.};
\node at (0.1,3) {.};
\node at (0.1,3.2) {.};
\end{tikzpicture}
\end{center}
\caption{The Quiver $Q_{\infty \times \infty}$}
\end{figure}
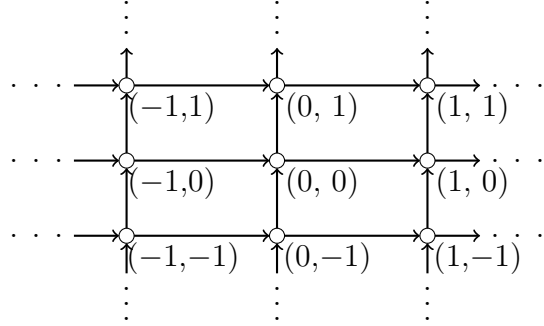
\end{example}

A \index{path}\emph{path} in a quiver $Q$ is a sequence $\tau = \rho_n \rho_{n-1} \cdots \rho_1$ of arrows such that $h(\rho_i)=t(\rho_{i+1})$ for each $1 \leq i \leq n-1$.  We define $t(\tau)=t(\rho_1)$ and $h(\tau)=h(\rho_n)$. The \index{path!algebra}\emph{path algebra} of $Q$ is the $\C$-algebra whose underlying vector space has for basis the set of paths in $Q$ and with multiplication given by:
\begin{equation*}
\tau_2 \cdot \tau_1 =
\begin{cases}
\tau_2 \tau_1, & \text{if } h(\tau_1)=t(\tau_2), \\
0, & \text{otherwise},
\end{cases}
\end{equation*}
where $\tau_2 \tau_1$ denotes the concatenation of the paths $\tau_1$ and $\tau_2$.  We denote the path algebra of $Q$ by $\C Q$.  For any vertex $x\in X$ we let $e_x$ denote the trivial path starting and ending at $x$ and with multiplication given by $e_x\tau=\delta_{h(\tau)x}\tau$ and $\tau e_x=\delta_{t(\tau)x}\tau$ for any path $\tau$.

Let $\ml V (X)$ denote the category of $X$-graded vector spaces over $\C$ with morphisms being linear maps that respect the grading.  We denote by $V(x)$ the $x$-component of an object $V \in \ml V(X)$ for any $x\in X$.  Then a \emph{representation} of a quiver $Q=(X,A)$ is an object $V \in \ml V(X)$ along with a linear map $V(\rho)\colon V(t(\rho))\to V(h(\rho))$ for every $\rho \in A$.  For a path $\tau=\rho_n\cdots \rho_1$, we define $V(\tau)=V(\rho_n) \cdots V(\rho_1)$.  If we let $\rep(Q)$ denote the category of representations of $Q$, then there is an equivalence of categories $\C Q\Md \cong \rep(Q)$.

A \emph{relation} in a quiver $Q$ is an expression of the form $r=\sum_{i=1}^k a_i \tau_i$, where $a_i \in \C$ and $\tau_i$ is a path in $Q$ for all $i=1,...,k$.  We say a representation $V\in \rep(Q)$ \emph{satisfies the relation} $r$ if $\sum_{i=1}^k a_i V(\tau_i)=0.$  If $R$ is a set of relations, we denote by $\rep(Q,R)$ the full subcategory of $\rep(Q)$ consisting of representations that satisfy the relations $R$.  There is an equivalence of categories $\C Q/I_R \Md \cong \rep(Q,R)$, where $I_R$ is the two-sided ideal of $\C Q$ generated by $R$.

%%%%%%%%%%%%%%%%%%%%%%%%%%%%%%%%%%%

\section{Quiver Morphisms}\label{quimorph}
In this section we establish some theory on quiver morphisms.  In particular, we recall the definition of a covering morphism of quivers.  We then discuss several functors, induced by such morphisms, between categories of quiver representations.  These functors have many nice properties that will play important roles in Section \ref{application}.

If $Q=(X,A,t,h)$ and $Q'=(X',A',t,h)$ then a \emph{quiver morphism} $\varphi \in \Hom(Q,Q')$ consists of a pair of maps $\varphi_1\colon X\to X'$ and $\varphi_2\colon A \to A'$ such that $\varphi_1 \circ t=t\circ\varphi_2$ and $\varphi_1 \circ h=h\circ\varphi_2$.

\begin{example}\label{morphf}
Let $Q_{\infty}$ denote the quiver $(\Z,\rho_i,\bar\rho_i)$, where $t(\rho_i)=i=h(\bar \rho_{i+1})$, and $h(\rho_i)=i+1=t(\bar \rho_{i+1})$, where indices are considered mod $s$.  If $Q_{\infty \times \infty}$ denotes the quiver from Example \ref{qinf}, consider the morphism $f \in \Hom(Q_{\infty \times \infty},Q_{\infty})$ given by $f_1(i,j)=i-j$, $f_2(\rho_1^{(i,j)})=\rho_{i-j}$, and $f_2(\rho_2^{(i,j)})=\bar \rho_{i-j}$ for all $(i,j) \in \Z\times \Z$.  It is easy to check that $f$ is indeed a quiver morphism.
\end{example}

\begin{example}\label{morphg}
For all $s \in \N$, define the quiver $\widehat Q_s=(\Z/s\Z, \rho_i,\bar \rho_i)$, where $t(\rho_i)=h(\bar \rho_{i+1})=i$ and $h(\rho_i)=t(\bar \rho_{i+1})=i+1.$  Note that when $s=0$ we retrieve the quiver $Q_{\infty}$ of Example \ref{morphf} above.  Let $m,n \in \Z$ with $\gcd(m,n)=1$.  Then we have $\gcd(m,m+n)=1$ and so for all $k \in \Z$ there is a unique integer $0\leq j_k <m+n$ such that $k \equiv j_k m \mod(m+n)$.  Then we have a morphism $g \in \Hom(Q_{m,n},\widehat Q_{m+n})$, where $Q_{m,n}$ denotes the quiver of Example \ref{qmn}, given by $g_1(k)=j_k$, $g_2(\rho^k_1)=\rho_{j_k}$, and $g_2(\rho^k_2)=\bar\rho_{j_k}$ for all $k \in \Z$.
\begin{figure}[h2]
\begin{center}
\begin{tikzpicture}
\draw (0,0) circle [radius=0.1];
\draw (0.9,-1.4) circle [radius=0.1];
\draw (-0.9,-1.4) circle [radius=0.1];
\draw[thick,->](-0.8,-1.48) to [out=320,in=220] (0.8,-1.48);
\draw[thick,->](0.8,-1.38) to [out=160,in=20] (-0.8,-1.38);
\draw[thick,->](-0.83,-1.29) to [out=20,in=270] (-0.07,-0.11);
\draw[thick,->](0.07,-0.12) to [out=270,in=160] (0.83,-1.29);
\draw[thick,->](-0.11,-0.05) to [out=200,in=90] (-0.89,-1.3);
\draw[thick,->](0.91,-1.28) to [out=90,in=340] (0.11,-0.05);
\node at (-1.2,-1.4) {1};
\node at (1.2,-1.4) {2};
\node at (0,0.4) {3};
\end{tikzpicture}
\end{center}
\caption{The quiver $\widehat Q_3$}
\end{figure}
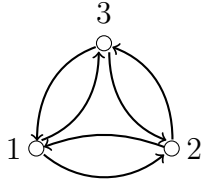
\end{example}

For the most part, we will be interested in quiver morphisms satisfying a certain property, known as \index{covering morphisms}\emph{covering morphisms}, which we introduce now.
\begin{defn}[Covering Morphism]
Let $Q=(X,A)$ and $Q'=(X',A')$ be quivers and let $\varphi \in \Hom(Q,Q')$.  Then $\varphi$ will be said to be a \emph{covering morphism} if for every $x \in X$ and every path $\tau'$ in $Q'$ such that $h(\tau')=\varphi(x)$ there exists a unique path $\tau$ in $Q$ such that $h(\tau)=x$ and $\varphi(\tau)=\tau'$.
\end{defn}
One easily verifies that the morphisms $g$ and $f$ of Examples \ref{morphg} and \ref{morphf} are covering morphisms.

Given a quiver morphism $\varphi \colon Q \to Q'$, there are three important functors one can use to study the relationship between the representations of the two quivers.
\begin{defn}\label{repfunctors}
Let $Q=(X,A)$ and $Q'=(X',A')$ be two quivers and let $\varphi \in \Hom(Q,Q')$.
\begin{enumerate}[(i)]
\item \index{restriction functor}The \emph{restriction functor} of $\varphi$ is the functor $\varphi^*\colon \Rep(Q')\to \Rep(Q)$ defined for all $V\in \Rep(Q')$ as follows:
\begin{enumerate}[(a)]
\item For any $x \in X$, $\varphi^*(V)(x)=V(\varphi(x))$.

\item For any $\rho \in A$, $\varphi^*(V)(\rho)=V(\varphi(\rho))$.
\end{enumerate}
If $V,U \in \Rep(Q)$ and $f\in \Hom_Q(V,U)$, then the morphism $\varphi^*(f)$ is defined by $\varphi^*(f)(x)=f(\varphi(x))$ for every $x \in X$.

\item If $\varphi$ is a covering morphism, then the \index{right!extension functor}\emph{right extension functor} of $\varphi$ is the functor $\varphi_* \colon \Rep(Q)\to \Rep(Q')$ defined for all $V \in \Rep(Q)$ as follows:
\begin{enumerate}[(a)]
\item For any $x' \in X'$, $\varphi_*(V)(x')=\prod_{x\in \varphi^{-1}(x')} V(x)$, where by convention we take the empty product to be the zero vector space.

\item For any $\rho' \in A'$, $\varphi_*(V)(\rho')=\prod_{\rho \in \varphi^{-1}(\rho')}V(\rho)$.
\end{enumerate}
If $V,U \in \Rep(Q')$ and $f \in \Hom_{Q'}(V,U)$, then the morphism $\varphi_*(f)$ is defined by $\varphi_*(f)(x')=\prod_{x\in \varphi^{-1}(x')} f(x)$.

\item If $\varphi$ is a covering morphism, then the \index{left!extension functor}\emph{left extension functor} of $\varphi$ is the functor $\varphi_! \colon \Rep(Q)\to \Rep(Q')$ defined for all $V \in \Rep(Q)$ as follows:
\begin{enumerate}[(a)]
\item For any $x' \in X'$, $\varphi_!(V)(x')=\bigoplus_{x\in \varphi^{-1}(x')} V(x)$, where by convention we take the empty coproduct to be the zero vector space.

\item For any $\rho' \in A'$, $\varphi_!(V)(\rho')=\bigoplus_{\rho \in \varphi^{-1}(\rho')}V(\rho)$.
\end{enumerate}
If $V,U \in \Rep(Q')$ and $f \in \Hom_{Q'}(V,U)$, then the morphism $\varphi_!(f)$ is defined by $\varphi_!(f)(x')=\bigoplus_{x\in\varphi^{-1}(x')} f(x)$.
\end{enumerate}
\end{defn}

We will use the above functors to study the representation theory of the quivers $Q_{m,n}$ and $Q_{\infty \times \infty}$ of Examples \ref{qmn} and \ref{qinf}.  It turns out that the functors in Definition \ref{repfunctors} are closely related.
\begin{prop}[{\cite[Theorem 4.1]{Enochs}}]\label{adjointthm}
Let $Q,Q'$ be two quivers and let $\varphi \in \Hom(Q,Q')$ be a covering morphism.  Then $\varphi^*$ is left adjoint to $\varphi_*$ and right adjoint to $\varphi_!$.
\end{prop}
\begin{proof}
In \cite[Theorem 4.1]{Enochs} it is shown explicitly that $\varphi^*$ is left adjoint to $\varphi_*$.  The proof that $\varphi^*$ is right adjoint to $\varphi_!$ is completely dual to the argument presented there.
\end{proof}

\begin{cor}\label{adjointcor}
The functors $\varphi^*$, $\varphi_*$, and $\varphi_!$ are additive.  Moreover, $\varphi^*$  is exact, $\varphi_*$ is left exact, and $\varphi_!$ is right exact.
\end{cor}
\begin{proof}
  The first statement follows from Theorem \ref{adjointthm} and the that fact that adjoint functors between additive categories are necessarily additive.  The second statement follows from Theorem \ref{adjointthm} and the fact that a left (resp.\ right) adjoint functor is always right (resp.\ left) exact.
\end{proof}

We will now focus on other properties of these functors.

\begin{lem}
Let $Q, Q'$ be quivers and let $\varphi \in \Hom(Q,Q')$ be a covering morphism.  Then both $\varphi_*$ and $\varphi_!$ are faithful.
\end{lem}
\begin{proof}
Let $V,U \in \Rep(Q)$ and consider the map
\begin{equation*}
(\varphi_!)_{VU}\colon \Hom_{\Rep(Q)}(V,U)\to \Hom_{\Rep(Q')}(\varphi_!(V),\varphi_!(U)).
\end{equation*}
We wish to show that this map is injective.  Since it is a group homomorphism, it is enough to consider the preimage of the morphism $0\colon \varphi_!(V)\to \varphi_!(U)$.  This morphism is defined by the linear maps $0(x')\colon \varphi_!(V)(x')\to \varphi_!(U)(x')$ for every $x' \in X'$.  So if $\varphi_!(\varsigma)=0$, then $\coprod_{x\in\varphi^{-1}(x')}\varsigma(x)=0$ for every $x'$, and it follows that $\varsigma(x)=0$ for every $x \in X$.  Hence $\sigma=0$.  To show that $\varphi_*$ is faithful is similar.
\end{proof}

\begin{lem}
Let $Q,Q'$ be quivers and let $\varphi \in \Hom(Q,Q')$ be a covering morphism.  Then both $\varphi_*$ and $\varphi_!$ are exact.
\end{lem}
\begin{proof}
By Corollary \ref{adjointcor}, we need only show that $\varphi_*$ is right exact and $\varphi_!$ is left exact.  Since the categories $\Rep(Q)$ and $\Rep(Q')$ are abelian, it is enough to show that $\varphi_*$ preserves cokernels and $\varphi_!$ preserves kernels.

Let $V,U \in \Rep(Q)$ and $\varsigma \in \Hom_{\Rep(Q)}(V,U)$.  The vector spaces of the representation $\varphi_*(\coker \sigma)$ are defined by $\varphi_*(\coker \varsigma)(x')=\prod_{x\in\varphi^{-1}(x')}U(x)/\im \varsigma(x)$.  However, the map $\varphi_*(\varsigma)(x')\colon \varphi_*(V)(x')\to \varphi_*(U)(x')$ is given by $\prod_{x\in\varphi^{-1}(x')}\varsigma(x)$.  Hence
\begin{align*}
\coker(\varphi_*(\varsigma))(x')&=\varphi_*(U)(x')/\im \varphi_*(\varsigma)(x')\\
&=\prod_{x\in\varphi^{-1}(x')} U(x)/\prod_{x\in\varphi^{-1}(x')}\im \varsigma(x) \\
&\cong \prod_{x\in\varphi^{-1}(x')} U(x)/\im \varsigma(x) \\
&=\varphi_*(\coker \varsigma)(x').
\end{align*}
It is now clear that for any arrow $\rho' \in A'$ the maps $\varphi_*(\coker \varsigma)(\rho')$ and $\coker(\varphi_*(\varsigma))(\rho')$ will be equal, and hence $\varphi_*$ preserves cokernels.

The representation $\varphi_!(\ker \varsigma) \in \Rep(Q)$ has vector spaces given by
\begin{equation*}
\varphi_!(\ker \varsigma)(x)=\bigoplus_{x\in\varphi^{-1}(x')} (\ker \varsigma)(x)=\bigoplus_{x\in\varphi^{-1}(x')} \ker(\varsigma(x))
\end{equation*}
for all $x \in X$.  On the other hand, the morphism $\varphi_!(\varsigma) \in \Hom_{\Rep(Q)}(\varphi_!(V),\varphi_!(U))$ is defined by $\varphi_!(\varsigma)(x')=\bigoplus_{x\in\varphi^{-1}(x')}\varsigma(x)$ and hence
\begin{equation*}
\ker(\varphi_!(\varsigma))(x')=\ker \left( \bigoplus_{x\in\varphi^{-1}(x')}\varsigma(x)\right)=\bigoplus_{x\in\varphi^{-1}(x')}\ker(\varsigma(x)).
\end{equation*}
Thus $\varphi_!(\ker \varsigma)(x')=\ker(\varphi_!(\varsigma))(x')$ for all $x'\in X'$.  Again, it is easy to see that for any $\rho' \in A'$ the maps $\varphi_!(\ker \varsigma)(\rho')$ and $\ker(\varphi_!(\varsigma))(\rho')$ will also coincide, and so $\varphi_!$ preserves kernels.
\end{proof}

Any quiver morphism $\varphi\colon Q \to Q'$ induces an algebra homomorphism $\varphi\colon \C Q \to \C Q'$, and so for any set of relations $R'$ in $Q'$ we can consider the preimage
\begin{equation*}
\varphi^{-1}(R')=\left \{\sum_{j=1}^k a_j \tau_j\ |\ \sum_{j=1}^k a_j \varphi(\tau_j) \in R'\right \}.
  \end{equation*}

  We will now show that if $\varphi$ is a covering morphism and $R'$ is a set of relations in $Q'$ then the functors $\varphi^*$,$\varphi_*$, and $\varphi_!$ restrict naturally to the subcategories of representations satisfying the relations $R'$ and $\varphi^{-1}(R')$.  First we recall that if $\varphi\colon Q\to Q'$ is a covering morphism and $\tau'$ is a path in $Q'$, then for any $y \in X$ with $\varphi(y)=h(\tau')$ there is a unique path $\tau$ in $Q$ ending at $y$ with $\varphi(\tau)=\tau'$.  We will denote this unique path $\tau$ by $\varphi^{-1}(\tau')_y$.
\begin{lem}\label{pathlemma}
Let $\varphi\colon Q \to Q'$ be a covering morphism.  Then for any path $\tau'$ in $Q'$ and any representation $V \in \Rep(Q)$ we have
\begin{align}
\varphi_!(V)(\tau')&=\bigoplus_{y: h(\tau')=\varphi(y)}V(\varphi^{-1}(\tau')_y), \text{ and }\label{pathlem1}\\
 \varphi_*(V)(\tau')&=\prod_{y: h(\tau')=\varphi(y)}V(\varphi^{-1}(\tau')_y). \label{pathlem2}
\end{align}
\end{lem}
\begin{proof}
We begin by noting that since $\varphi$ is a covering morphism, for any path $\tau'$ in $Q'$ there is a bijection between paths $\tau$ in $Q$ with $\varphi(\tau)=\tau'$ and vertices $x \in X$ with $\varphi(x)=h(\tau')$ given by $x \leftrightarrow \varphi^{-1}(\tau')_x$.  Let $\tau'$ be a path in $Q'$.  We will proceed by induction on the length of $\tau'$, which we denote by $n$.  If $n=0$, then the path $\tau'$ is a trivial path at some vertex $y'$ in $Q'$.  Then for any vertex $y \in\varphi^{-1}(y')$, $\varphi^{-1}(\tau')_y$ is the trivial path at $y$.  Hence we have
\begin{equation*}
\varphi_!(V)(\tau')=\id_{\varphi_!(V)(y')}=\bigoplus_{y\in \varphi^{-1}(y')} \id_{V(y)} = \bigoplus_{y:h(\rho')=\varphi(y)}V(\varphi^{-1}(\rho')_y).
\end{equation*}
 Now let $n\geq1$ and write $\tau'=\alpha'\rho'$, where $\alpha'$ is a path in $Q'$ of length $n-1$ and $\rho'\in A'$.  Then if we assume the result is true for the path $\alpha'$, we have
\begin{equation}\label{pathdecomp}
\varphi_!(V)(\alpha')\varphi_!(V)(\rho')=\bigoplus_{y:h(\alpha')=\varphi(y)}V(\varphi^{-1}(\alpha')_y) \bigoplus_{x:h(\rho')=\varphi(x)} V(\varphi^{-1}(\rho')_x).
\end{equation}
Since $\varphi$ is a covering morphism, $|\{\alpha\ |\ \varphi(\alpha)=\alpha'\}|=|\{\tau\ |\ \varphi(\tau)=\tau'\}|$ as both sets are in bijection with the set of vertices in $X$ that get mapped to $h(\alpha')=h(\tau')$.  Moreoever, for any $y \in X$ with $\varphi(y)=h(\tau')$, if we let $x=t(\varphi^{-1}(\alpha')_y)$ then we have $\varphi^{-1}(\tau')_y=\varphi^{-1}(\alpha')_y \varphi^{-1}(\rho')_x$.  It follows that each $\varphi^{-1}(\tau')_y$ shows up exactly once in the decomposition \eqref{pathdecomp} above.  On the other hand, if $x \in X$ is a vertex such that the arrow $\varphi^{-1}(\rho')_x$ cannot be extended to a path which maps to $\tau'$ under $\varphi$, then in the decomposition \eqref{pathdecomp} $V(x)$ is mapped to 0.  We can therefore ignore such components, and we arrive at \eqref{pathlem1}.  To prove that \eqref{pathlem2} holds is similar.
\end{proof}

\begin{prop}\label{respectrelations}
Let $Q,Q'$ be quivers, $R'$ a set of relations in $Q'$, and $\varphi\colon Q \to Q'$ a covering morphism.
\begin{enumerate}[(i)]
\item If $V \in \Rep(Q',R')$, then $\varphi^*(V) \in \Rep(Q,\varphi^{-1}(R'))$.

\item If $V \in \Rep(Q,\varphi^{-1}(R'))$ then $\varphi_!(V),\varphi_*(V)\in \Rep(Q',R')$.
\end{enumerate}
\end{prop}
\begin{proof}
\begin{asparaenum}[(i)]
\item Let $\sum_{j=1}^k a_j\tau'_j \in R'$.  Then since $V \in \Rep(Q',R')$, we have $\sum_{j=1}^k a_j V(\tau'_j)=0$.  Hence for any $y \in X$ with $\varphi(y)=h(\tau'_j)$ we have
\begin{equation*}
\sum_{j=1}^k a_j \varphi^*(V)(\varphi^{-1}(\tau'_j)_y)=\sum_{j=1}^k a_j V(\tau'_j)=0.
\end{equation*}
It follows that $\varphi^*(V)\in \Rep(Q,\varphi^{-1}(R'))$.

\item Once again, suppose $\sum_{j=1}^k a_j\tau'_j \in R'$.  Let $V \in \Rep(Q,\varphi^{-1}(R'))$ and let $H=\{h(\tau_j')\ |\ j=1,\dots,k\}$.  Then using Lemma~\ref{pathlemma} we get:
\begin{equation*}
\sum_{j=1}^k a_j \varphi_!(V)(\tau'_j)=\sum_{j=1}^k a_j \bigoplus_{y\in\varphi^{-1}(\tau'_j)} V(\varphi^{-1}(\tau'_j)_y)=\bigoplus_{y\in\varphi^{-1}(H)} \sum_{j=1}^k a_j V(\varphi^{-1}(\tau'_j)_y)=0.
\end{equation*}
It follows that $\varphi_!(V)\in \Rep(Q',R')$.  To show that $\varphi_*(V) \in \Rep(Q',R')$ is similar.
\end{asparaenum}
\end{proof}

We will denote by $\varphi^*_{R'}$ the restriction of the functor $\varphi^*$ to the subcategory $\Rep(Q',R')$.  Similarly, we denote by $\varphi_!^{R'}$ and $\varphi_*^{R'}$ the restrictions of the functors $\varphi_!$ and $\varphi_*$ to the subcategory $\Rep(Q,\varphi^{-1}(R'))$.  For any set of relations $R'$ in $Q'$ the functors $\varphi^*_{R'}$, $\varphi_!^{R'}$, and $\varphi^{R'}_*$ are all additive, exact, and both $\varphi_!^{R'}$ and $\varphi_*^{R'}$ are faithful.  Moreover, Theorem \ref{adjointthm} implies that $\varphi^*_{R'}$ is left adjoint to $\varphi_*^{R'}$ and right adjoint to $\varphi_!^{R'}$.  \label{repfunctors2}

%%%%%%%%%%%%%%%%%%%%%%%%%%%%%%%%%%%

\section{Representations of $L_{\mu}$}\label{application}

In this section we will obtain some results concerning the representation theory of the Lie algebras $L_\mu$.  First we will establish relationships between the modified enveloping algebras introduced in Section \ref{liealg} and the path algebras of the quivers $Q_{m,n}$ and $Q_{\infty\times\infty}$ introduced in Section \ref{quiv}.  More specifically, we will show that the modified enveloping algebras are isomorphic to the path algebras of $Q_{m,n}$ (for $\mu=\frac{m}{n}$, where $m,n \in \Z$ with $n\neq 0$) and $Q_{\infty\times\infty}$ ($\mu \in \C \setminus \Q$) modulo certain ideals.  We will then be able to use the theory of quiver morphisms of Section \ref{quimorph} to relate these quivers to simpler ones, and we conclude with some statements regarding representations of the Lie algebras $L_{\mu}$.

\subsection{Relation to the Lie Algebras $L_\mu$}

Consider the linear map $\varphi$ determined by:
\begin{align}
\varphi:\C Q_{m,n} &\to \tilde{U}_{m,n}, \nonumber \\
\rho_{j_{s}}^{k+\sigma(j_{1})+\dots+\sigma(j_{s-1})}\cdots \rho_{j_{1}}^{k} &\mapsto \alpha_{j_{s}}\cdots\alpha_{j_{1}}a_{k}.
\end{align}

\begin{lem} \label{homo}
The map $\varphi$ is a homomorphism of algebras.
\end{lem}
\begin{proof}
Since $\varphi$ is linear, it suffices to show that it commutes with the multiplication. Let 
\begin{equation*}
\rho_{j_{s}}^{k+\sigma(j_{1})+\dots+\sigma(j_{s-1})}\cdots\rho_{j_{1}}^{k},  \rho_{i_{r}}^{\ell+\sigma(i_{1})+\dots+\sigma(i_{r-1})}\cdots\rho_{i_{1}}^{\ell} \in \C Q_{m,n}.
\end{equation*}  
Then
\begin{align*}
&(\rho_{j_{s}}^{k+\sigma(j_{1})+\dots+\sigma(j_{s-1})}\cdots\rho_{j_{1}}^{k})\cdot(\rho_{i_{r}}^{\ell+\sigma(i_{1})+\dots+\sigma(i_{r-1})}\cdots\rho_{i_{1}}^{\ell}) \\
& \quad = \left\{
\begin{array}{cl}
\rho_{j_{s}}^{k+\sigma(j_{1})+\dots+\sigma(j_{s-1})}\cdots\rho_{j_{1}}^{k}\rho_{i_{r}}^{\ell+\sigma(i_{1})+\dots+\sigma(i_{r-1})}\cdots\rho_{i_{1}}^{\ell} & \text{if } \ell+\sigma(i_{1})+\dots+\sigma(i_{r})=k \\
0 & \text{otherwise }
\end{array} \right. \\
\implies &\varphi\left((\rho_{j_{s}}^{k+\sigma(j_{1})+\dots+\sigma(j_{s-1})}\cdots\rho_{j_{1}}^{k})\cdot(\rho_{i_{r}}^{\ell+\sigma(i_{1})+\dots+\sigma(i_{r-1})}\cdots\rho_{i_{1}}^{\ell})\right)\\
&\quad = \left\{
\begin{array}{cl}
\alpha_{j_{s}}\cdots\alpha_{j_{1}}\alpha_{i_{r}}\cdots\alpha_{i_{1}}a_\ell & \text{if } \ell+\sigma(i_{1})+\dots+\sigma(i_{r})=k, \\
0 & \text{otherwise. }
\end{array} \right.
\end{align*}
On the other hand,
\begin{align*}
\varphi& \left(\rho_{j_{s}}^{k+\sigma(j_{1})+\dots+\sigma(j_{s-1})}\cdots\rho_{j_{1}}^{k}\right)\varphi\left(\rho_{i_{r}}^{\ell+\sigma(i_{1})+\dots+\sigma(i_{r-1})}\cdots\rho_{i_{1}}^{\ell}\right) \\
&\quad =(\alpha_{j_{s}}\cdots\alpha_{j_{1}}a_k)(\alpha_{i_{r}}\cdots\alpha_{i_{1}}a_{\ell}) \\
&\quad =\alpha_{j_{s}}\cdots\alpha_{j_{1}}\alpha_{i_{r}}a_{k-\sigma(i_{r})}\alpha_{i_{r-1}}\cdots\alpha_{i_{1}}a_\ell \\
&\quad =\alpha_{j_{s}}\cdots\alpha_{j_{1}}\alpha_{i_{r}}\alpha_{i_{r-1}}a_{k-\sigma(i_{r})-\sigma(i_{r-1})}\alpha_{i_{r-2}}\cdots\alpha_{i_{1}}a_\ell \\
&\quad \: \: \vdots \\
&\quad =\alpha_{j_{s}}\cdots\alpha_{j_{1}}\alpha_{i_{r}}\cdots\alpha_{i_{1}}a_{k-\sigma(i_{r})-\dots-\sigma(i_{1})}a_\ell \\
&\quad =\left\{
\begin{array}{cl}
\alpha_{j_{s}}\cdots\alpha_{j_{1}}\alpha_{i_{r}}\cdots\alpha_{i_{1}}a_\ell & \text{if } \ell=k-\sigma(i_{1})-\dots-\sigma(i_{r}), \\
0 & \text{otherwise. }
\end{array} \right.
\end{align*}
And hence
\begin{align*}
& \varphi\left((\rho_{j_{s}}^{k+\sigma(j_{1})+\dots+\sigma(j_{s-1})}\cdots\rho_{j_{1}}^{k})\cdot(\rho_{i_{r}}^{\ell+\sigma(i_{1})+\dots+\sigma(i_{r-1})}\cdots\rho_{i_{1}}^{\ell})\right)\\
& \qquad = \varphi\left(\rho_{j_{s}}^{k+\sigma(j_{1})+\dots+\sigma(j_{s-1})}\cdots\rho_{j_{1}}^{k}\right)\varphi\left(\rho_{i_{r}}^{\ell+\sigma(i_{1})+\dots+\sigma(i_{r-1})}\cdots\rho_{i_{1}}^{\ell}\right). \qedhere
\end{align*}
\end{proof}
We now establish a precise relationship between the path algebra $\C Q_{m,n}$ and the modified enveloping algebra $\widetilde{U}_{m,n}$.
\begin{prop} \label{ratiso}
For any $m,n \in \Z^*$ there is an isomorphism of algebras $\widetilde{U}_{m,n}\cong \C Q_{m,n}/I^{m,n}$, where $I^{m,n}$ is the two-sided ideal generated by the elements $\rho_{1}^{k+n}\rho_{2}^{k}-\rho_{2}^{k+m}\rho_{1}^{k}$ for $k\in\Z$.
\end{prop}
\begin{proof}
We claim that $I^{m,n} \subseteq \ker\varphi$.  Indeed:
\begin{align}
\varphi(\rho_{1}^{k+n}\rho_{2}^{k}-\rho_{2}^{k+m}\rho_{1}^{k})&=\varphi(\rho_{1}^{k+n}\rho_{2}^{k})-\varphi(\rho_{2}^{k+m}\rho_{1}^{k}) \nonumber \\
&=\alpha_1\alpha_2a_k-\alpha_2\alpha_1a_k \nonumber\\
&=\alpha_1\alpha_2a_k-\alpha_1\alpha_2a_k \tag{by \eqref{modrel}}\\
&=0 \nonumber 
\end{align}
Thus $\varphi$ induces a morphism
\begin{equation*}
\bar{\varphi}:\C Q_{m,n}/I^{m,n} \to \tilde{U}, \quad
x+I \mapsto \varphi(x) \qquad \forall x \in \C Q_{m,n}.
\end{equation*}
We will also consider the linear map determined by:
\begin{align*}
\psi:\widetilde{U} &\to \C Q_{m,n}/I^{m,n} \\
{\alpha_{1}}^{r}a_k{\alpha_{2}}^{s} &\mapsto \rho_{1}^{k+(r-1)m}\cdots\rho_{1}^{k}\rho_{2}^{k-n}\cdots\rho_{2}^{k-sn}+I.
\end{align*}
We will show that $\psi\bar{\varphi}$ and $\bar{\varphi}\psi$ are identity maps. Seeing as both maps are linear, it will suffice to show that this is the case for basis elements. Let $\rho_{j_{s}}^{k+\sigma(j_{1})+\dots+\sigma(j_{s-1})}\cdots\rho_{j_{1}}^{k}+I \in \C Q/I^{m,n}$.  Then
\begin{align*}
(\psi\bar{\varphi})\left(\rho_{j_{s}}^{k+\sigma(j_{1})+\dots+\sigma(j_{s-1})}\cdots\rho_{j_{1}}^{k}+I\right)&=\psi\left(\bar{\varphi}(\rho_{j_{s}}^{k+\sigma(j_{1})+\dots+\sigma(j_{s-1})}\cdots\rho_{j_{1}}^{k}+I)\right) \\
&=\psi\left(\varphi(\rho_{j_{s}}^{k+\sigma(j_{1})+\dots+\sigma(j_{s-1})}\cdots\rho_{j_{1}}^{k})\right) \\
&=\psi(\alpha_{j_{s}}\cdots\alpha_{j_{1}}a_k) \\
&=\psi(\alpha_1^r\alpha_2^ta_k), \quad \text{for some } r+t=s \\
&=\psi(\alpha_1^ra_{k+tn}\alpha_2^t) \\
&=\rho_1^{k+tn+(r-1)m}\cdots \rho_1^{k+tn}\rho_2^{k+(t-1)n}\cdots \rho_2^k +I \\
&=\rho_{j_{s}}^{k+\sigma(j_{1})+\dots+\sigma(j_{s-1})}\cdots\rho_{j_{1}}^{k}+I \\
\\
\implies \psi\bar{\varphi}&=\id:\C Q_{m,n}/I^{m,n} \to \C Q_{m,n}/I^{m,n}.
\end{align*}
In the last line of the computation we used the relation $\rho_1^{k+n} \rho_2^k - \rho_2^{k+m} \rho_1^k \in I^{m,n}$ to reorder the terms into a new member of the same equivalence class.

Now let ${\alpha_1}^{r}a_k{\alpha_2}^{s} \in \widetilde{U}$.  Then
\begin{align*}
(\bar{\varphi}\psi)({\alpha_{1}}^{r}a_k{\alpha_{2}}^{s})&=\bar{\varphi}\left(\psi({\alpha_{1}}^{r}a_k{\alpha_{2}}^{s})\right) \\
&=\bar{\varphi}\left(\rho_{1}^{k+(r-1)m}\cdots\rho_{1}^{k}\rho_{2}^{k-n}\cdots\rho_{2}^{k-sn}+I\right) \\
&=\varphi\left(\rho_{1}^{k+(r-1)m}\cdots\rho_{1}^{k}\rho_{2}^{k-n}\cdots\rho_{2}^{k-sn}\right) \\
&={\alpha_{1}}^{r}{\alpha_{2}}^{s}a_{k-sn} \\
&={\alpha_{1}}^{r}a_k{\alpha_{2}}^{s} \\
\\
\implies \bar{\varphi}\psi &=\id:\widetilde{U} \to \widetilde{U}.
\end{align*}

Combining this result with Lemma \ref{homo}, we see that $\bar{\varphi}$ is a bijective homomorphism of algebras, which completes the proof.
\end{proof}
By Proposition \ref{ratiso}, there is an equivalence of categories $\widetilde U_{m,n}\Md \cong \rep(Q_{m,n}, R_{m,n})$, where $R_{m,n}$ is the set of relations of the form $\rho_{1}^{k+n}\rho_{2}^{k}-\rho_{2}^{k+m}\rho_{1}^{k}$, $k\in\Z$, in $Q_{m,n}$.  One can show similarly that when $\mu \in \C\backslash \Q$ there is an equivalence of categories $\widetilde U_\mu \Md \cong \Rep(Q_{\infty\times\infty},R_{\infty\times\infty)}$, where $R_{\infty\times\infty}$ denotes the set of relations of the form $\rho_{1}^{i,j+1}\rho_{2}^{ij}-\rho_{2}^{i+1,j}\rho_{1}^{ij}$, $i,j\in\Z$, in $Q_{\infty\times\infty}$.

\subsection{The Quiver $Q_{m,n}$}\label{rationalfunctor}

Let $\mu=\frac{n}{m}$, $\gcd(m,n)=1$.   We define $\widehat Q_s$ as in Example \ref{morphg}.  If we also define $\widehat R_{m+n}=\{ \bar \rho_{i+1} \rho_i - \rho_{i-1}\bar \rho_{i} \ |\ i \in \Z/s\Z \}$, then we can study the representations of $\widetilde U_{m,n}$ by relating the category $\rep(Q_{m,n},R_{m,n})$ to the category $\rep(\widehat Q_{m+n}, \widehat R_{m+n})$.  This is an interesting connection, and moreover much is known about the category $\rep(\widehat Q_{m+n}, \widehat R_{m+n})$, see \cite{SavStatMech} for example.

Let $V=(V(k),V(\rho^k_i))\in \rep(Q_{m,n})$ and let $j \in \Z/(m+n)\Z$.  If $k \in \Z$, we will write $k \equiv j \mod(m+n)$ simply as $k \equiv j$.  Then if $g \in \Hom(Q_{m,n},\widehat Q_{m+n})$ is the morphism described in Example \ref{morphg}, the representation $g_!(V)$ has vector spaces given by
\begin{equation*} \textstyle
g_!(V)(j)=\bigoplus_{k \equiv jm}V_k.
\end{equation*}

The linear maps of the representation are given by $g_!(V)(\rho_j)=\bigoplus_{k \equiv jm} V(\rho_1^k)$ and $g_!(V)(\bar\rho_j)=\bigoplus_{k \equiv jm} V(\rho_2^k)$.  Note that $g_!(V)(\rho_j)$ maps $g_!(V)(j)$ to $g_!(V)(j+1)$, and $g_!(V)(\bar\rho_j)$ maps $g_!(V)(j)$ to $g_!(V)(j-1)$.  Moreover, if $V,U \in \Rep(Q_{m,n})$ and $\varphi \in \Hom_{Q_{m,n}}(V,U)$ then $g_!(\varphi)=\{g_!(\varphi)(j)\colon g_!(V)(j)\to g_!(U)(j)\}$, where $g_!(\varphi)(j)=\bigoplus_{k\equiv jm}\varphi_k$.

Let $V,U \in \rep(Q_{m,n})$.  Then $g_!(V) \ncong g_!(U) \Rightarrow V \ncong U$ since $g_!$ is a functor.  Further, it follows from Corollary \ref{adjointcor} that if $g_!(V)$ is indecomposable then $V$ is indecomposable, since additive functors preserve finite coproducts.

Note that the preimage of the relations $\widehat R_{m+n}$ in $\widehat Q_{m+n}$ under $g_!$ are exactly the relations $R_{m,n}$, that is, $g_!^{-1}(\widehat R_{m+n})=R_{m,n}$.  Thus, by Proposition \ref{respectrelations}, we can restrict the functor $g_!$ to the subcategory $\rep(Q_{m,n}, R^{m,n})$ of $\rep(Q_{m,n})$ to get a functor $g_!^{\widehat R_{m+n}}: \rep(Q_{m,n}, R_{m,n}) \to \rep(\widehat Q_{m+n}, \widehat R_{m+n})$.  Furthermore, since $g_!$ is additive, faithful, and exact, so too is $g_!^{\widehat R_{m+n}}$, and we can therefore relate the categories $\rep(\widetilde U_{m,n})$ and $\rep(\widehat Q_{m+n}, \widehat R_{m+n})$.  It is natural to ask whether or not $g_!^{\widehat R_{m+n}}$ gives an equivalence of categories, and it turns out that this is true only when $\mu = -1$,  in which case $g_!$ is the identity functor.  When $\mu \neq -1$, the functor $g_!^{\widehat R_{m+n}}$ is neither full nor essentially surjective, as the following examples illustrate:

\begin{example}
Let $V \in \rep(Q_{m,n},R_{m,n})$ be the representation given by $V_0=V_{n+m}=\C$, and $V_i=0$ for all other $i \in \Z$.  The endomorphism space of $V$ is $\Hom_{Q_{m,n}}(V,V) \cong \Hom(\C, \C) \oplus \Hom(\C,\C) \cong \C^2$.  The representation $g_!^{\widehat R_{m+n}}(V) \in \rep(\widehat Q_{m+n}, \widehat R_{m+n})$ is the representation such that $g_!^{\widehat R_{m+n}}(V)_0= \C^2$, and all other vector spaces are zero.  Thus $\Hom_{\widehat Q_{m+n}}(g_!^{\widehat R_{m+n}}(V),g_!^{\widehat R_{m+n}}(V)) \cong \Hom(\C^2, \C^2) \cong \Mat_{2 \times 2}(\C)$.  Since the dimension of the endomorphism space of $g_!^{\widehat R_{m+n}}(V)$ is greater than the dimension of the endomorphism space of $V$, the map $g_{!VV}^{\widehat R_{m+n}}$ is not surjective.  Hence $g_!^{\widehat R_{m+n}}$ is not full.
\end{example}

\begin{example}
Let $U \in \rep(Q_{m,n},R_{m,n})$. Then if $U$ is finite dimensional, there are only finitely many nonzero $U_k$, and so there exists some $t \in \Z$ such that \[U(\rho_1^{k+tm}\rho_1^{k+(t-1)m}\cdots \rho_1^k)=0\] for any $k \in \Z$.  Then
$\bigoplus_{k \equiv jm} U(\rho_1^{k+tm}\rho_1^{k+(t-1)m}\cdots \rho_1^k)=0$, and so there is a path in the representation $g_!^{\widehat R_{m+n}}(U) \in \rep(\widehat Q_{m+n}, \widehat R_{m+n})$ that acts by zero.  Consider the representation $V \in \rep(\widehat Q_{m+n}, \widehat R_{m+n})$ defined by $(V_i, x_i, \bar x_i)=(\C, \lambda, 1)$ for all $i \in \Z/(m+n)\Z$, where $\lambda \in \C^*$.  Suppose there were some representation $U \in \rep(Q_{m,n}, R^{m,n})$ such that $g_!^{\widehat R_{m+n}}(U) \cong V$.  Then since $V$ is finite dimensional, $U$ must also be finite dimensional.  However, since there are no paths in the representation $V$ that act by zero, this is a contradiction.  Hence $g_!^{\widehat R_{m+n}}$ is not essentially surjective.
\end{example}

While the category $\Rep(Q_{m,n},R_{m,n})$ can be quite difficult to study in general, if we restrict our attention to representations which are supported on certain numbers of vertices, we can obtain some results about the representation theory of the Lie algebras $L_{\mu}$.  For any $m,n \in \Z$ and any $a,b \in \Z$ with $a<b$, we will denote by $\ml C^{m,n}_{a,b}$ the full subcategory of $\widetilde U_{m,n}\Md$ consisting of modules $V$ such that $V_k=0$ whenever $k<a$ or $k>b$ (see \eqref{weightspace1}).  
\begin{theo}\label{m=1}
  Let $m\in\Z$ and let $a,b\in\Z$ be such that $0\leq b-a\leq m$.   Then $\ml C^{m,1}_{a,b}$ is of tame representation type.
\end{theo}
\begin{proof}
  Note that any representation $V\in \Rep(Q_{m,1})$ which is supported on at most $m+1$ consecutive vertices automatically satisfies the relations $R_{m,1}$ in a trivial way.  Thus we may identify representations $V\in \ml C^{m,n}_{a,b}$ with representations of a quiver whose underlying graph is an extended Dynkin diagram of type $\hat A_m$.  Quivers with underlying graphs of type $\hat A_m$ are of tame representation type, so the result then follows from the equivalence $\widetilde{U}_{m,1}\Md \cong \Rep(Q_{m,1},R_{m,1})$.
\end{proof}
\begin{theo}\label{mneq1}
  Let $m,n \in\Z$, $\gcd(m,n)=1$, $n\neq 1$.  Let $a,b \in Z$ be such that $0\leq b-a\leq m$.  Then $\ml C^{m,n}_{a,b}$ is of finite representation type.  
\end{theo}
\begin{proof}
  As previously noted, any such representation $V$ trivially satisfies the relations $R_{m,n}$.  Thus we may identify $V$ with a representation of a quiver whose underlying graph is a Dynkin diagram of type $A$, and so the result follows from the equivalence $\widetilde U_{m,n}\Md \cong \Rep(Q_{m,n},R_{m,n})$.  
\end{proof}

\subsection{The Quiver $Q_{\infty \times \infty}$}

When $\mu \in \C \setminus \Q$ we are interested in representations of the quiver $Q_{\infty \times \infty}$ introduced in Example \ref{qinf}.  In order to study the representations of $Q_{\infty \times \infty}$ we will relate the category $\rep(Q_{\infty \times \infty})$ to the category $\rep (Q_{\infty})$.  Here $Q_{\infty}$ denotes the quiver from Example \ref{morphf}. We will then be able to relate the category $\rep(Q_{\infty \times \infty}, R_{\infty \times \infty})$ to the category $\rep(Q_{\infty}, R_{\infty})$, where $R_{\infty}=\{\bar\rho_{i+1} \rho_i - \rho_{i-1}\bar \rho_{i} \ |\ i \in \Z\}$. We will obtain a relationship between these two categories which is similar to the relationship we studied in Section \ref{rationalfunctor}.  The representation theory of the quiver $Q_{\infty}$ subject to relations $R_{\infty}$ has been studied in \cite[Section 4]{SavEuc}.

\begin{remark}
Given any representation $V\in \Rep(Q_{\infty\times\infty},R_{\infty\times\infty})$ and any $i\in \Z$, we can consider the representation of the quiver of type $A_\infty$ given by the $i^{\text{th}}$ row of $V$, that is, the representation $V_i \in \Rep(A_\infty)$ with $V_i(j)=V(i,j)$.  Then the fact that $V$ satisfies the relations $R_{\infty\times\infty}$ implies that the collection $\{V(\rho_2^{ij})\ |\ j\in\Z\}$ defines a morphism of $A_\infty$ representations $V_i\to V_{i+1}$.  Thus we may think of elements of $\Rep(Q_{\infty\times\infty},R_{\infty\times\infty})$ as chains of representations of the quiver of type $A_\infty$.
\end{remark}

Let $V$ be a representation of the quiver $Q_{\infty \times \infty}$ and let $f\in \Hom(Q_{\infty \times \infty},Q_\infty)$ be the morphism described in Example \ref{morphf}.  Then the functor $f_!$ has vector spaces given by
\begin{equation*} \textstyle
f_!(V)(k):=\bigoplus_{i-j=k} V(i,j).
\end{equation*}

The linear maps between these spaces are given by $f_!(V)(\rho_k)=\bigoplus_{i-j=k}V(\rho^{ij}_1)$ and $f_!(V)(\bar \rho_k)=\bigoplus_{i-j=k}V(\rho_2^{ij})$.  Note that $f_!(V)(\rho_k)$ maps $f_!(V)(k)$ to $f_!(V)(k+1)$ and $f_!(\bar\rho_k)$ maps $f_!(V)(k)$ to $f_!(V)(k-1)$.  The functor $f_!$ acts on morphisms of $\Rep(Q_{\infty\times\infty})$ as follows: if $\varphi=\{\varphi(i,j)\}$ is a morphism between two representations $V$ and $U$ of $Q_{\infty \times \infty}$, where $\varphi(i,j)\colon V(i,j)\to U(i,j)$, then $f_!(\varphi)=\{f_!(\varphi)(k)\}$, where $f_!(\varphi)(k)=\bigoplus_{i-j=k}\varphi(i,j)$.

Once again Corollary \ref{adjointcor} implies that $f_!$ is an additive functor.  If two objects $f_!(V), f_!(U) \in \rep(Q_{\infty})$ are non-isomorphic, the objects $V,U \in \rep(Q_{\infty \times \infty})$ must be non-isomorphic.  Also, if an object $f_!(V) \in \rep(Q_{\infty})$ is indecomposable, then the object $V \in \rep(Q_{\infty \times \infty})$  is also indecomposable, as in the rational case.

The relations $R_\infty$ are the Gelfand-Ponomarev relations in $Q_\infty$. The functor $f_!$ can be restricted to the subcategory $\rep(Q_{\infty \times \infty}, R_{\infty \times \infty})$ of $\rep(Q_{\infty \times \infty})$ to yield a functor $f_!^{R_\infty}:\rep(Q_{\infty \times \infty}, R_{\infty \times \infty}) \to \rep(Q_{\infty}, R_{\infty})$, and this restricted functor is additive, faithful, and right exact. However, the following examples illustrate that $f_!^{R_\infty}$ is neither full nor essentially surjective:
\begin{example}
Consider the following representation $V \in \rep(Q_{\infty \times \infty}, R_{\infty \times \infty})$:
\begin{center}
$\begin{CD}
0 @>>> \C \\
@AAA @AAA \\
\C @>>> 0,
\end{CD}$
\end{center}
where all vector spaces not shown in the diagram are zero, and $V(0,0)=V(1,1)=\C$.  The endomorphism space of $V$ is given by $\Hom_{Q_{\infty \times \infty}}(V,V) \cong \Hom(\C, \C) \oplus \Hom(\C, \C)\cong \C^2$.  The object $f_!^{R_\infty}(V)$ is the representation of $Q_{\infty}$ such that $f_!^{R_\infty}(V)_0=\C^2$, and all other vertices are $0$.  The endomorphism space of this representation, however, is $\Hom_{Q_\infty}(f_!^{R_\infty}(V), f_!^{R_\infty}(V)) \cong \Hom(\C^2,\C^2) \cong \Mat_{2 \times 2}(\C)$.  Since the dimension of $\Hom_{Q_\infty}(f_!^{R_\infty}(V),f_!^{R_\infty}(V))$ is greater than the dimension of $\Hom_{Q_{\infty \times \infty}}(V,V)$, the induced functor $f_{!VV}^{R_\infty}$ is not surjective, and hence $f_!^{R_\infty}$ is not full.
\end{example}

\begin{example}
Next, consider the representation $V \in \rep(Q_{\infty}, R_{\infty})$ given by $(V(i), V(\rho_i), V(\bar \rho_i))= (\C, \lambda, 1)$ for all $i \in \Z$, where $\lambda \in \C$ is nonzero.  Suppose $V \cong f_!^{R_\infty}(U)$ for some $U \in \rep(Q_{\infty \times \infty}, R_{\infty \times \infty})$.  Recall that the vertical maps $U(\rho^{ij}_2)$ of the representation $U$ correspond to the leftward maps $V(\bar \rho_i)$ of $V$.  Since each $V(i)$ is one-dimensional, and each $V(i)$ maps to each $V(i-1)$ through the identity map, all nonzero $U(i,j)$ must lie along the same column.  Relabelling if necessary, we may assume it's the first column.  Then we must have $U(1,j) \cong V(j)$ and $U(\rho^{1j}_2) \cong 1$.  A similar argument shows that all nonzero $U(i,j)$ must lie along the first \emph{row}, with $U(i,1) \cong V(i)$ and $U(\rho_1^{i1}) \cong \lambda$.  Clearly, no such $U$ exists, and hence $f_!^{R_\infty}$ is not essentially surjective.
\end{example}

We will now consider an example which shows that the representation theory of the quiver $Q_{\infty\times\infty}$ is at least of tame type.  To do this, we will show that there exists a family of pairwise nonisomorphic indecomposable representations of $Q_{\infty\times\infty}$ which depend upon a continuous parameter.

\begin{example}
  Let $V_\lambda \in \Rep(Q_{\infty\times\infty},R_{\infty\times\infty})$ denote the following representation, where all vector spaces and maps not displayed are assumed to be zero.
  \begin{equation*}
  \begin{CD}
 \C @>1>> \C @. \\
 @A1AA @AA(1\ 1)A @. \\
 \C @>>(1\ 0)^T> \C^2 @> (1\ \lambda) >> \C \\
 @. @A (0\ 1)^T AA @AA1A \\
  @. \C @>>1> \C
  \end{CD}
  \end{equation*}
  We will assume that $V_\lambda(0,0)=\C^2$, and label all other vertices accordingly.  First we will show that $V_\lambda$ is indecomposable for all $\lambda \in \C$.  Suppose $V_\lambda=U\oplus W$.  Then we may assume $U(-1,1)=\C$.  Since $V_\lambda (\rho_1^{-1,1})=V_\lambda(\rho_2^{-1,0})=1$, we must have $U(0,1)=U(-1,0)=\C$.  We then have $(1\ \lambda)(1\ 0)^T(\C)=\subseteq U(0,1)$, and it follows that $U(0,1)=\C$.  But then $U(0,-1)=U(1,-1)=\C$ since $V_\lambda(\rho_1^{0,-1})=V_\lambda(\rho_2^{1,-1})=1$.  Finally, we have $(1\ 0)^T(\C)\subseteq U(0,0)$ and $(0\ 1)^T(\C)\subseteq U(0,0)$, and we conclude that $U=V_\lambda$, so $V_\lambda$ is indecomposable.  

  Now suppose $V_\lambda \cong V_\mu$.  Then there exists an invertible 2$\times$2 matrix $A$ and nonzero complex numbers $z_1,z_2,z_3,z_4 \in \C$ such that the following equations hold:
  \begin{align*}
    (1\ 0)^T z_1&=A(1\ 0)^T, \\
    (0\ 1)^T z_2&=A(0\ 1)^T, \\
    z_3 (1\ 1)&=(1\ 1)A, \\
    (1\ \mu)A&=z_4(1\ \lambda).
  \end{align*}
  The first two equations insist that $A$ is a diagonal matrix.  The third equation then implies that it is a scalar matrix, and then the fourth equation forces $\lambda=\mu$.  Hence when $\lambda\neq \mu$, $V_\lambda$ and $V_\mu$ are nonisomorphic.  One can show in a similar manner that the images of these representations under the functor $f_!^{R_\infty}$ gives a family of indecomposable pairwise nonisomorphic representations in $\Rep(Q_{\infty},R_\infty)$.
\end{example}

While we have seen that it is neither full nor essentially surjective, the functor $f^{R_\infty}_!$ can still be used to study the category $\Rep(Q_{\infty\times\infty},R_{\infty\times\infty})$.  First, we note that the group $\Z$ acts on the vertices and arrows of $Q_{\infty\times\infty}$ via
\begin{equation*}
z\cdot (i,j)=(i+z,j+z), \quad z\cdot \rho_1^{ij}=\rho_1^{(i+z)(j+z)},\quad z\cdot\rho_2^{ij}=\rho_2^{(i+z)(j+z)}.
\end{equation*}
Given any representation $V\in \Rep(Q_{\infty\times\infty},R_{\infty\times\infty})$, we denote by $V^{(z)}$ the representation obtained from $V$ by twisting with the action of $\Z$.  More precisely, the representation $V^{(z)}$ is defined by
\begin{equation*}
  V^{(z)}(i,j)=V(i+z,j+z), \quad V^{(z)}(\rho_1^{ij})=V(\rho_1^{(i+z)(j+z)}), \quad V^{(z)}(\rho_2^{ij})=V(\rho_2^{(i+z)(j+z)}).
\end{equation*}
\begin{lem}\label{surjon5}
Let $a,b \in \Z$ be integers such that $0<b-a\leq 4$.  Let $V \in \Rep(Q_{\infty},R_\infty)$ be a finite dimensional representation such that $V(k)=0$ whenever $k<a$ or $k>b$.  Then $V$ is isomorphic to $f_!^{R_\infty}(U)$ for some $U\in \Rep(Q_{\infty\times\infty},R_{\infty\times\infty})$, which is unique up to translation $U \mapsto U^{(z)}$ by the group $\Z$.
\end{lem}
\begin{proof}
Since the functor $f_!^{R_\infty}$ is additive, we may assume that $V$ is indecomposable.  Any indecomposable representation $V\in \Rep_{\mathrm{fd}}(Q_\infty, R_\infty)$ such that $V(k)=0$ whenever $k<a$ or $k>b$ is supported on at most 5 vertices, and hence may be thought of as a representation of the preprojective algebra of the quiver of type $A_5$.  The lemma then follows from \cite[Lemma 9.1]{GLS}, which states a similar result in the case of preprojective algebras of type $A_n$ for $2\leq n\leq 5$.
\end{proof}
The translation $V \mapsto V^{(z)}$ by $\Z$ on representations of $Q_{\infty\times\infty}$ induces an action of $\Z$ on the collection of isomorphism classes of representations of $U_\mu$ admitting a weight space decomposition via the equivalences $\Rep(Q_{\infty\times\infty},R_{\infty\times\infty})\cong \wtrep_\gamma(U_\mu)\cong \widetilde{U}_\mu\Md$.  We then have the following result:
\begin{prop}
For $a,b \in \Z$ with $0\leq b-a\leq 3$, there are a finite number of $\Z$-orbits of isomorphism classes of indecomposable $\widetilde U_\mu$-modules $V$ such that $V_{ij}=0$ whenever $i-j<a$ or $i-j>b$.
\end{prop}
\begin{proof}
By \cite[Theorem 4.3]{SavEuc}, there are a finite number of isomorphism classes of indecomposable modules $V \in \Rep(Q_{\infty},R_\infty)$ such that $V(k)=0$ for $k<a$ or $k>b$.  The proposition then follows from the equivalence $\Rep(Q_{\infty\times\infty}, R_{\infty\times\infty})\cong \widetilde U_\mu \Md$ and Lemma \ref{surjon5}.
\end{proof}
\begin{cor}\label{irratreps}
Let $A$ be a finite subset of $\Z$ with the property that $A$ does not contain any five consecutive integers.  Then there are a finite number of $\Z$-orbits of isomorphism classes of indecomposable $\widetilde U_\mu$-modules $V$ such that $V_{ij}=0$ whenever $i-j \notin A$.
\end{cor}

\bibliography{thesisbib}
\bibliographystyle{amsplain}

\end{document}